\documentclass[12pt,english]{article}
\usepackage{babel}
\usepackage[cp1250]{inputenc}
\usepackage[T1]{fontenc}
\usepackage{amsfonts}
\usepackage{amsmath}
\usepackage{amsthm}
\usepackage{mathrsfs}
\usepackage{hyperref}
\usepackage{enumerate}
\usepackage{graphicx}
\usepackage{pictex}

\newcommand{\eqq}[2]{\begin{equation}  #1  \label{#2} \end{equation}    }
\newcommand*{\norm}[1]{\left\Vert{#1}\right\Vert}
\newcommand*{\abs}[1]{\left\vert{#1}\right\vert}

\newcommand*{\izj}{\int_{0}^{1}}

\newcommand*{\p}{\partial}
\newcommand*{\poch}{\frac{\p}{\p x}}

\newcommand*{\ia}{I^{\alpha}}
\newcommand*{\al}{\alpha}
\newcommand*{\uz}{u_{0}}

\newcommand*{\ve}{\varepsilon}

\def\Re{\operatorname {Re}}
\def\Im{\operatorname {Im}}
\def\arg{\operatorname{arg}}

\newcommand{\hd}{\hspace{0.2cm}}
\newcommand{\no}{\noindent}
\newcommand{\m}[1]{\mbox{#1}}

\newtheorem{rem}{{\textbf {Remark}}}
\newtheorem{lem}{{\textbf {Lemma}}}
\newtheorem{prop}{{\textbf {Proposition}}}
\newtheorem{theorem}{\textbf {Theorem}}
\newtheorem{coro}{\textbf  {Corollary} }
\newtheorem{de}{\textbf  {Definition} }

\newcommand{\izx}{\int_{0}^{x}}

\newcommand{\da}{D^{\alpha}}

\newcommand{\ld}{L^{2}(0,1)}

\newcommand{\jgja}{\frac{1}{\Gamma(1-\al)}}

\newcommand{\ca}{c_{\al}}
\newcommand{\ba}{b_{\al}}
\newcommand{\dd}{\mathcal{D}_{\al}}

\begin{document}
\title{\bf An analytic semigroup generated by a fractional differential operator}

\author{Katarzyna Ryszewska\footnote{Department of Mathematics and Information
Sciences, Warsaw University of Technology, ul. Koszykowa 75, 00-662 Warsaw,
Poland, E-mail address:
 K.Ryszewska@mini.pw.edu.pl}}

\date{}

\maketitle

\abstract{ }
We study a space-fractional diffusion problem, where the non-local diffusion flux involves the Caputo derivative of the diffusing quantity. We prove the unique existence of regular solutions to this problem by means of the semigroup theory. We show that the operator defined as divergence of product of a positive function with the Caputo derivative is a generator of analytic semigroup.
\vspace{0.4cm}

\no Keywords: space-fractional diffusion; Caputo derivative; analytic semigroup theory.
\vspace{0.2cm}

\no AMS subject classifications (2010):  47D60, 35R11, 35K20, 26A33.

\section{Introduction }
In the paper we consider the following initial-boundary value problem
\begin{equation}\label{aa}
 \left\{ \begin{array}{ll}
u_{t} - \frac{\p}{\p x}(p(x) \da u) = f & \textrm{ in } (0,1) \times (0,T),\\
u_{x}(0,t) = 0, \ \ u(1,t) = 0 & \textrm{ for  } t \in (0,T), \\
u(x,0) = u_{0}(x) & \textrm{ in } (0,1), \\
\end{array} \right. \end{equation}
where, for $\al \in (0,1)$, the operator $\da$ denotes the fractional Caputo derivative with respect to spatial variable and $p$ is a positive lipschitz function. The motivation for studying (\ref{aa}) originates from modelling of sub-surface water motion.
In paper \cite{V2} the author considered the model of infiltration of water into heterogeneous soils. Due to the presence of media heterogeneity, the author proposed representing the hydraulic flux in terms of fractional derivative. The similar approach for modelling the non-locality in space was presented in \cite{V1}. This paper provides the overview of one-phase Stefan problems which exhibits anomalous behaviour. The author introduced the model where the diffusive operator is in the divergence form and the flux is the Caputo derivative of the temperature. In the present paper we consider the simplified problem (\ref{aa}). We would like to emphasise that from the modelling point of view, it is important that (\ref{aa}) is a balance low.

Our goal is to study the basic solvability problem for equation (\ref{aa}), by the semigroup approach. At this moment we would like to present a broader context. It is worth to mention here, that in the paper \cite{stoch} the authors studied an array of fractional diffusion equations, for a variety of boundary conditions, but they constructed only the $C_{0}$-semigroup. Besides, we note that the problem (\ref{aa}) was not explicitly studied in \cite{stoch}. Another paper that presents the probabilistic point of view on space-fractional problems is \cite{stoch2}, where the authors consider equations with time-fractional Caputo derivative and non-local space operators. A completely different approach for solving (\ref{aa}) for $p\equiv 1$ with zero Dirichlet boundary conditions is employed in \cite{visco}, where the authors obtained the viscosity solutions. Further discussion was made in \cite{stoch3} and \cite{VR} where the authors compare the problems with diffusive flux modeled by the Caputo and the Riemann-Liouville derivative and carry a numerical analysis.

In the present paper we will present the results concerning solvability of (\ref{aa}) by means of the semigroup theory. At first, we will focus our attention on the case where $p \equiv 1$. We will describe the domain of $\poch \da$ in terms of Sobolev spaces and as a final result we will construct an analytic semigroup. Precisely, we will show that the operator $\poch \da$, considered on the domain
\[
\dd := \{u \in H^{1+\al}(0,1) :u_{x} \in {}_{0}H^{\al}(0,1),   u(1)=0\}
\]
generates an analytic semigroup. Here, by $H^{\al}(0,1)$ we mean the fractional Sobolev space (see \cite[definition 9.1]{Lions}) and the subspace ${}_{0}H^{\al}(0,1)$ will be introduced in Proposition \ref{eq}. Subsequently, we will extend our result for the case of any strictly positive lipschitz continuous $p$. We would like to emphasise that, developing the theory of analytic semigroups for the systems of type (\ref{aa}) seems to be especially important if we notice that the operator $\poch \da$ is not self-adjoint, thus we can not expect that its eigenfunctions generate the basis of any natural Hilbert space.

The paper is organized as follows. In chapter 2 we give preliminary results concerning fractional operators. Chapter 3 is devoted to the proof of the main result. Namely, we will show that $\poch \da$ is a generator of $C_{0}$-semigroup of contractions which can be extended to the analytic semigroup (Theorem \ref{anal}). In chapter 4, we extend the results of the previous chapter to the case of operator $\poch (p(x)\da)$, where $p$ is a positive lipschitz function (Theorem \ref{gen}). In chapter 5, we will present the basic applications of this result to solvability of (\ref{aa}). At last, for the sake of completeness, in Appendix we collect the general results from calculus that we used in the paper.
\section{Preliminaries from fractional calculus}
Before we will establish the results, we will introduce definitions of fractional operators and recall some of their properties. For more comprehensive studies on general properties of fractional operators we refer to a standard literature \cite{Kilbas}, \cite{Samko}.\\
\begin{de}\label{fracdef}
Let $L, \al > 0$. For $f \in L^{1}(0,L)$ we introduce the fractional integral $\ia$ by the formula
\[
\ia f (x) = \frac{1}{\Gamma(\al)}\izx (x-p)^{\al-1}f(p)dp.
\]
If $\al \in (0,1)$ and $f$ is regular enough we may define the Riemann-Liouville fractional derivative
\[
\p^{\al}f(x) =  \poch I^{1-\al}f(x)= \jgja \poch \izx (x-p)^{-\al} f(p)dp
\]
and the Caputo fractional derivative
\[
\da f(x) = \poch (I^{1-\al}[f(x)-f(0)])=  \jgja  \izx (x-p)^{-\al}[f(p)-f(0)]dp.
\]
We note that if $f \in AC[0,T]$, then $\da f$ may be equivalently written in the form
\[
\da f(x) = I^{1-\al}f'(x) = \jgja  \izx (x-p)^{-\al}f'(p)dp.
\]
We will also make use of the formal adjoint operators, that is
\[
I^{\al}_{-}f(x) = \frac{1}{\Gamma(\al)}\int_{x}^{L}(p-x)^{\al-1}f(p)dp,
\]
\[
\p^{\al}_{-}f(x) = - \poch (I^{1-\al}_{-}f)(x) \m{ and } D^{\al}_{-}f(x) = - \poch I^{1-\al}_{-}[f(x)-f(L)].
\]
\end{de}
It is clear that for absolutely continuous functions the foregoing fractional differential operators are well defined. Further discussion on the correctness of these definitions is performed in \cite{Kilbas} and \cite{Samko}. In Proposition \ref{eq}, we will introduce the definition of the domain for the Riemann-Liouville fractional derivative in terms of Sobolev spaces which was established in \cite{Y}.
\no Before we will pass to that result, we will present a simple proposition which gives us the formula for the superposition of $\p^{\al}$ and~$\da$.

\begin{prop}\cite[Proposition 6.5]{KY}\label{skladanie}
Let $L > 0$. For $\al, \beta \in (0,1)$ such that $\al+\beta \leq 1$ and $u \in AC[0,L]$ we have $\p^{\al}D^{\beta}u = D^{\al+\beta}u.$
\end{prop}
\no The next proposition provides the energy estimate for the Riemann-Liouville fractional derivative.
\begin{prop}\cite[Proposition 6.10]{KY}
If $w \in AC[0, L]$ then for any $\al \in (0,1)$ the following equality holds,
\[
\int_{0}^{L} \p^{\al}w(x) \cdot w(x)dx = \frac{\al}{4} \int_{0}^{L} \int_{0}^{L} \frac{\abs{w(x)-w(p)}^{2}}{\abs{x-p}^{1+\al}}dp dx
\]
\[
+\frac{1}{2\Gamma(1-\al)}\int_{0}^{L} [(L-x)^{-\al} + x^{-\al}]\abs{w(x)}^{2} dx.
\]
Hence, there exist a positive constant $c$ which depends only on $\al,L$, such that
\eqq{
\int_{0}^{L} \p^{\al}w(x) \cdot w(x)dx  \geq c \norm{w}_{H^{\frac{\al}{2}}(0,L)}^{2}
}{AH}
and in particular
\[
\int_{0}^{L} \p^{\al}w(x) \cdot w(x)dx  \geq \frac{L^{-\al}}{2\Gamma(1-\al)}\int_{0}^{L}\abs{w(x)}^{2} dx.
\]
\end{prop}

Now, we will introduce the characterization of the domain of the Riemann- Liouville derivative in the $L^{2}$-framework. It will be essential for further considerations.
\no At first let us define the following functional spaces
\[
_{0}H^{\al}(0,1)=\left\{
\begin{array}{lll}
H^{\al}(0,1)  & \m{ for } & \al\in (0,\frac{1}{2}), \\
\{ u \in H^{\frac{1}{2}}(0,1): \hd \int_{0}^{1}\frac{|u(x)|^{2}}{x}dx<\infty \} & \m{ for } & \al=\frac{1}{2}, \\
\{ u \in H^{\al}(0,1): \hd u(0)=0\} & \m{ for } & \al\in (\frac{1}{2},1)
\end{array}
\right.
\]\
and
\[
^{0}H^{\al}(0,1)=\left\{
\begin{array}{lll}
H^{\al}(0,1)  & \m{ for } & \al\in (0,\frac{1}{2}), \\
\{ u \in H^{\frac{1}{2}}(0,1): \hd \int_{0}^{1}\frac{|u(x)|^{2}}{1-x}dx<\infty \} & \m{ for } & \al=\frac{1}{2}, \\
\{ u \in H^{\al}(0,1): \hd u(1)=0\} & \m{ for } & \al\in (\frac{1}{2},1).
\end{array}
\right.
\]\
We set $\norm{u}_{_{0}H^{\al}(0,1)} = \norm{u}_{^{0}H^{\al}(0,1)} = \norm{u}_{H^{\al}(0,1)}$ for $\al \neq \frac{1}{2}$ and
\[
\norm{u}_{_{0}H^{\frac{1}{2}}(0,1)}  =\left(\norm{u}_{H^{\frac{1}{2}}(0,1)}^{2} + \izj \frac{\abs{u(x)}^{2}}{x} dx\right)^{\frac{1}{2}},
\]
\[
\norm{u}_{^{0}H^{\frac{1}{2}}(0,1)} =\left(\norm{u}_{H^{\frac{1}{2}}(0,1)}^{2} + \izj \frac{\abs{u(x)}^{2}}{1-x} dx\right)^{\frac{1}{2}}.
\]
The following proposition is the extended version of \cite[Theorem 2.1]{Y} which can be found in the Appendix of \cite{KY}.
\begin{prop}\label{eq}
For $\al\in [0,1]$ the operators
$\ia: L^{2}(0,1)\longrightarrow {}_{0}H^{\al}(0,1)$ and $\p^{\al}:{}_{0}H^{\al}(0,1)\longrightarrow L^{2}(0,1)$ are isomorphism and the following
inequalities hold
$$\ca^{-1} \| u \|_{_{0}H^{\al}(0,1)}
\leq \| \p^{\al} u \|_{L^{2}(0,1)}\leq \ca
\| u \|_{_{0}H^{\al}(0,1)} \hd \m{ for } u \in {}_{0}H^{\al}(0,1),
$$

$$
\ca^{-1} \| \ia f  \|_{_{0}H^{\al}(0,1)}
\leq \|  f \|_{L^{2}(0,1)}\leq \ca \| \ia f
\|_{_{0}H^{\al}(0,1)} \hd \m{ for } f  \in L^{2}(0,1).
$$
Analogously, by the change of variables $x\mapsto 1-x$, we obtain that the operators
$\ia_{-}: L^{2}(0,1)\longrightarrow {}^{0}H^{\al}(0,1)$  and $\p^{\al}_{-}:{}^{0}H^{\al}(0,1)\longrightarrow L^{2}(0,1)$ are isomorphism and there hold the inequalities
$$\ca^{-1} \| u \|_{^{0}H^{\al}(0,1)}
\leq \| \p^{\al}_{-} u \|_{L^{2}(0,1)}\leq \ca
\| u \|_{^{0}H^{\al}(0,1)} \hd \m{ for } u \in {}^{0}H^{\al}(0,1),
$$

$$
\ca^{-1} \| \ia_{-} f  \|_{^{0}H^{\al}(0,1)}
\leq \|  f \|_{L^{2}(0,1)}\leq \ca \| \ia_{-} f
\|_{^{0}H^{\al}(0,1)} \hd \m{ for } f  \in L^{2}(0,1).
$$\
Here $\ca$ denotes a positive constant dependent on $\al$.
\end{prop}

\begin{coro}\label{eqc}
For $\al, \beta \in (0,1)$ such that $0< \al+\beta \leq 1$ we have
\[
I^{\beta}: {}_{0}H^{\al}(0,1)\rightarrow {}_{0}H^{\al+\beta}(0,1).
\]
\end{coro}
\begin{proof}
It is an easy consequence of Proposition \ref{eq}. If $u \in {}_{0}H^{\al}(0,1),$ then there exists $w \in L^{2}(0,1)$ such that $u = I^{\al}w$.
Then,
\[
I^{\beta} u = I^{\beta} I^{\al}w = I^{\al+\beta}w \in {}_{0}H^{\al+\beta}(0,1).
\]
\end{proof}
We finish this section with two propositions which provide us an extension of $\ia$ and $\p^{\al}$ into wider functional spaces. The similar reasoning to the one carried in Proposition \ref{ie} may be found in \cite[Lemma 5]{DPG}.

\begin{prop}\label{ie}
For $\al \in (0,\frac{1}{2})$ the operators $I^{\al}$ and $I^{\al}_{-}$ can be extended to bounded and linear operators from $H^{-\al}(0,1):=(H^{\al}_{0}(0,1))'$ to $L^{2}(0,1)$.
\end{prop}
\begin{proof}
We will prove the claim only for $I^{\al}$ while the proof for $I^{\al}_{-}$ is analogous.
By the Fubini theorem for $u,v \in L^{2}(0,1)$ we obtain
\eqq{
\left(I^{\al}u,v\right) = \left(u,I^{\al}_{-}v\right).
}{fub}
Thus, using Proposition \ref{eq} we may estimate,
\[
\abs{\left(I^{\al}u,v\right) } \leq \norm{I^{\al}_{-}v}_{H^{\al}(0,1)}\norm{u}_{(H^{\al}(0,1))'} \leq \ca \norm{v}_{L^{2}(0,1)}\norm{u}_{H^{-\al}(0,1)},
\]
where we used the fact that for $\al < \frac{1}{2}$ we have $(H^{\al}_{0}(0,1))' = (H^{\al}(0,1))'$. The last inequality finishes the proof.
\end{proof}

\begin{prop}\label{de}
For $\al \in (0,\frac{1}{2})$ the operators $\p^{\al}$ and $\p^{\al}_{-}$ can be extended to bounded and linear operators from $L^{2}(0,1)$ to $H^{-\al}(0,1)$.
\end{prop}
\begin{proof}
As in the previous proposition, we will prove the statement only for $\p^{\al}$, because for $\p^{\al}_{-}$ the proof is analogous.
Let us assume that $f,v \in H^{\al}(0,1)$. (We recall that for $\al \in (0, \frac{1}{2})$ the space $H^{\al}(0,1)$ coincides with $^{0}H^{\al}(0,1)$ and $_{0}H^{\al}(0,1))$. Then, from Proposition \ref{eq}, there exist $g \in L^{2}(0,1)$ such that $\p^{\al}f = g$ and $w \in H^{\al}(0,1)$ such that $v = \ia_{-} w$. Thus, we have
\[
\left(\p^{\al}f,v\right) = \left(g,I^{\al}_{-}w\right) = \left(I^{\al}g, w\right) =\left(f,\p^{\al}_{-}v\right).
\]
Making use of Proposition \ref{eq} one more time, we may estimate
\[
\abs{\left(\p^{\al}f,v\right) } \leq \norm{f}_{L^{2}(0,1)}\norm{\p^{\al}_{-}v}_{L^{2}(0,1)} \leq \ca \norm{f}_{L^{2}(0,1)}\norm{v}_{H^{\al}(0,1)},
\]
Thus $\p^{\al}f \in (H^{\al}(0,1))'$ which coincides with $H^{-\al}(0,1)$ for $\al \in (0,\frac{1}{2})$ and the proof is finished.
\end{proof}

\section{Operator $\poch \da$ as a generator of an analytic semigroup}

In this section, in order to make an argument transparent, we deal with the case $p\equiv 1$. The case of arbitrary $p$ is slightly more technical and it will be considered in the next chapter. We will proceed as follows. Firstly, we will characterize the domain of $\poch \da$ in $L^{2}(0,1)$. Then, we will show that $\poch \da$ generates a $C_{0}$~-semigroup of contractions. Finally, we will prove that, by appropriate estimate of the resolvent operator, this semigroup may be extended to an analytic semigroup on a sector of complex plane.

We may note that, just by definition $\poch \da u = \poch I^{1-\al} u_{x} = \p^{\al}u_{x},$ whenever one of sides of the identity is meaningful.
By Proposition \ref{eq} the domain of $\p^{\al}$ in $\ld$ coincides with $ {}_{0}H^{\al}(0,1).$ Thus, we may consider the domain of $\poch \da$ as $\{u\in H^{1+\al}(0,1): u_{x} \in {}_{0}H^{\al}(0,1)\}$. Taking into account the boundary condition in~(\ref{aa}) we finally define the domain of $\poch \da$ as
\eqq{D(\poch \da) \equiv \dd := \{u \in H^{1+\al}(0,1) :u_{x} \in {}_{0}H^{\al}(0,1), \hd  u(1)=0\}.}{dziedzina}
We may equip $\dd$ with the graph norm
\[
\norm{f}_{\dd} = \norm{f}_{H^{1+\al}(0,1)} \m{ for } \al \in (0,1) \setminus \{\frac{1}{2}\}
\]
and
\[
\norm{f}_{\dd} = \left( \norm{f}_{H^{\frac{3}{2}}(0,1)}^{2} + \izj \frac{\abs{u_{x}(x)}^{2}}{x}dx\right)^{\frac{1}{2}} \m{ for }\al = \frac{1}{2}.
\]
\begin{theorem}\label{C0}
Operator $\poch \da: \dd \subseteq\ld \rightarrow \ld$ generates a $C_{0}$ semigroup of contractions.
\end{theorem}
\begin{proof}

We will prove Theorem \ref{C0} by applying the Lumer-Philips theorem \cite[Ch.~1, Theorem 4.3]{Pazy}.
Firstly, we see that $\poch \da$ is densely defined.
In order to satisfy assumptions of Lumer-Philips theorem we need to show in addition that $-\poch \da$ is accretive and that $R(I - \poch\da) = \ld.$ In order to prove that $-\poch \da$ is accretive we take $u\in \dd$, using integration by parts and Proposition \ref{skladanie}, we get
\[
\Re \left(-\poch \da u, u\right) = -\Re \izj (\poch \da u)(x) \cdot \overline{u(x)}dx
\]
\[
= \izj \da \Re u(x)\cdot \poch \Re u(x)dx + \izj \da \Im u(x)\cdot \poch \Im u(x)dx
\]
\[
 = \izj \da \Re u(x)\cdot \p^{1-\al}\da \Re u(x)dx + \izj \da \Im u(x)\cdot \p^{1-\al}\da \Im u(x)dx.
 \]
 Since $u_{x} \in  {}_{0}H^{\al}(0,1),$ then from Corollary \ref{eqc} we know that $\da u = I^{1-\al}u_{x} \in {}_{0}H^{1}(0,1).$ We may apply inequality (\ref{AH}) with $w=\da \Re u$ and $w= \da \Im u$ to obtain
\[
\Re \left(-\poch \da u, u\right)  \geq \ca \norm{\da u}_{H^{\frac{1-\al}{2}}(0,1)}^{2} \geq \ca \norm{\p^{\frac{1-\al}{2}}\da u}_{L^{2}(0,1)}^{2}
\]
\eqq{
 = \ca \norm{D^{\frac{1+\al}{2}}\ u}_{L^{2}(0,1)}^{2},
}{koer0}
where in the second inequality we used Proposition \ref{eq} and the equality follows from Proposition \ref{skladanie}. Here $\ca > 0$ denotes a generic constant dependent on $\al$.\\
Now, we would like to show that $R(I - \poch\da) = \ld.$ In fact, we are able to show something more. We will state the result in the next lemma.
\begin{lem}
For every $\lambda \in \mathbb{C}$ belonging to the sector
\eqq{
\vartheta_{\al}:=\{z \in \mathbb{C}:\abs{\arg z} \leq \frac{\pi (\al+1)}{2}\} \cup \{0\}
}{sektor}
there holds
\[
R(\lambda I - \poch\da) = \ld.
\]
\end{lem}
\begin{proof}
To prove the lemma we fix  $g \in \ld$ and $\lambda$ belonging to $\vartheta_{\al}$. We must prove that there exists $u \in \dd$ such that
\eqq{
\lambda u - \poch \da u = g.
}{zwyczajne}
We would like to calculate the solution directly. To that end, we will firstly solve equation (\ref{zwyczajne}) with the arbitrary boundary condition $u(0) = \uz \in \mathbb{C}.$ Then, we will choose $\uz$ which will guarantee the zero condition at the other endpoint of the interval.  We note that if we search for a solution in $\{f \in H^{1+\al}(0,1): f_{x} \in {}_{0}H^{\al}(0,1) \}$, then equation (\ref{zwyczajne}) is equivalent to
\eqq{
u = \uz  +\lambda I^{\al+1}u - I^{\al+1}g.
}{calkowe1}
Indeed, if we apply $\ia$ to both sides of (\ref{zwyczajne}), recall that $\poch \da u = \p^{\al}u_{x}$ and assume that $u_{x} \in {}_{0}H^{\al}(0,1)$, we obtain
\[
u_{x} = \lambda \ia u - \ia g.
\]
Integrating this equality we arrive at (\ref{calkowe1}). On the other hand, if we assume that $u\in L^{2}(0,1)$ solves (\ref{calkowe1}), then by Proposition \ref{eq} it automatically belongs to $\{f \in H^{1+\al}(0,1): f_{x} \in {}_{0}H^{\al}(0,1) \}$ and to obtain (\ref{zwyczajne}) it is enough to apply  $\p^{\al}\poch$ to~(\ref{calkowe1}).\\
Thus, we are going to solve (\ref{calkowe1}).
We apply to (\ref{calkowe1}) the operator $I^{\al+1}$ and we obtain
\[
u(x) = \uz - (I^{\al+1}g)(x) + \lambda (I^{\al+1} \uz)(x) + \lambda^{2}(I^{2(\al+1)}u)(x) - \lambda (I^{2(\al+1)}g)(x).
\]
Iterating this procedure $n$ times we arrive at
\eqq{
u(x) = \uz\sum_{k=0}^{n}(\lambda^{k} I^{k(\al+1)}1) - \sum_{k=0}^{n}\lambda^{k}(I^{(k+1)(\al+1)}g)(x) + \lambda^{n+1}(I^{(n+1)(\al+1)}u)(x).
}{don}
We will show, that the last expression tends to zero as $n\rightarrow \infty$.
Indeed, we may note that, since we search for the solutions in $H^{1+\al}(0,1) \subseteq L^{\infty}(0,1)$ and due to the presence of the $\Gamma$-function in the denominator we have
\[
\abs{\lambda^{n}(I^{n(\al+1)}u)(x)} \leq \norm{u}_{L^{\infty}(0,1)}\frac{\abs{\lambda}^{n} x^{(\al +1)n}}{\Gamma((\al+1)n+1)} \leq \frac{\norm{u}_{L^{\infty}(0,1)}\abs{\lambda}^{n} }{\Gamma((\al+1)n+1)}\rightarrow 0 \textrm{ as } n\rightarrow \infty
\]
for each $\lambda \in \mathbb{C}$ uniformly with respect to $x\in [0,1]$. Thus, passing to the limit with $n$ in (\ref{don}) we obtain the formula
\eqq{
u(x) = \uz\sum_{k=0}^{\infty}(\lambda^{k} I^{k(\al+1)}1) - \sum_{k=0}^{\infty}\lambda^{k}(I^{(k+1)(\al+1)}g)(x).
}{szereg1}
We will show that both series in (\ref{szereg1}) are uniformly convergent. We may estimate
\[
\abs{\lambda^{k}(I^{(\al+1)(k+1)}g)(x)} \leq \frac{\abs{\lambda}^{k}}{\Gamma((\al+1)(k+1))}\izx (x-s)^{(\al+1)(k+1)-1} \abs{g(s)}ds
\]
\[
 \leq \norm{g}_{L^{2}(0,1)}\frac{\abs{\lambda}^{k}}{\Gamma((\al+1)(k+1))}\left(\izx (x-s)^{2(\al+1)(k+1)-2}ds\right)^{\frac{1}{2}}
\]
\[
  \leq \norm{g}_{L^{2}(0,1)}\frac{\abs{\lambda}^{k}}{\Gamma((\al+1)(k+1))}\frac{x^{(\al+1)(k+1)-\frac{1}{2}}}{\sqrt{2(\al+1)(k+1)-1}}.
\]
If we denote the last expression by $a_{k}$ we may calculate
\[
\frac{a_{k+1}}{a_{k}} = \frac{\abs{\lambda} x^{\al+1}\sqrt{2(\al+1)(k+1)-1}\Gamma((\al+1)(k+1))}{\sqrt{2(\al+1)(k+2)-1}\Gamma((\al+1)(k+1) + (\al+1))}
\]
and so
\[
\frac{a_{k+1}}{a_{k}} \leq \abs{\lambda} x^{\al+1} \frac{B((\al+1),(k+1)(\al+1))}{\Gamma(\al+1)} \rightarrow 0 \textrm { as } k\rightarrow \infty
\]
for every $\lambda \in \mathbb{C}$ and uniformly with respect to $x\in [0,1]$ and from comparison criterion and the d'Alembert criterion the series is uniformly convergent.  With the first series in (\ref{szereg1}) we may deal even more easily. Now, we would like to compute the sum of the series. By direct calculations, we see that
\eqq{
\sum_{k=0}^{\infty}\lambda^{k}(I^{(\al+1)k}1)(x) = E_{\al+1}(\lambda x^{\al+1}),
}{mla}
where by $E_{\al+1}$ we denote the Mittag-Leffler function. For the definition of Mittag-Leffler function we refer to Proposition \ref{MLp} from the Appendix.
To calculate the sum of the second series, we apply the definition of fractional integral (see Definition \ref{fracdef})
\[
\sum_{k=0}^{\infty}\lambda^{k}(I^{(\al+1)(k+1)}g)(x) = \sum_{k=0}^{\infty}\lambda^{k}\izx g(s)\frac{(x-s)^{(\al+1)k+\al}}{\Gamma((\al+1)k+\al+1)}ds.
\]
In order to interchange the order of integration and summation, we will firstly consider the finite sum and then we will pass to the limit,
\[
\sum_{k=0}^{\infty}\lambda^{k}\izx g(s)\frac{(x-s)^{(\al+1)k+\al}}{\Gamma((\al+1)k+\al+1)}ds
= \lim_{n\rightarrow \infty}\sum_{k=0}^{n}\lambda^{k}\izx g(s)\frac{(x-s)^{(\al+1)k+\al}}{\Gamma((\al+1)k+\al+1)}ds
\]
\[
=\lim_{n\rightarrow \infty}\izx g(s)\sum_{k=0}^{n}\lambda^{k}\frac{(x-s)^{(\al+1)k+\al}}{\Gamma((\al+1)k+\al+1)}ds.
\]
We would like to apply the Lebesgue dominated convergence theorem, thus we need to indicate the majorant. We may estimate as follows
\[
\abs{g(s)\sum_{k=0}^{n}\lambda^{k}\frac{(x-s)^{(\al+1)k+\al}}{\Gamma((\al+1)k+\al+1)}} \leq \abs{g(s)}\sum_{k=0}^{\infty}\frac{\abs{\lambda}^{k}}{\Gamma((\al+1)k+\al+1)}
\]
\[
= \abs{g(s)}E_{\al+1,\al+1}(\abs{\lambda})
\]
and the last function is integrable because $g \in L^{2}(0,1)$. Hence, applying the Lebesgue dominated convergence theorem we arrive at
\[
\sum_{k=0}^{\infty}\lambda^{k}(I^{(\al+1)(k+1)}g)(x) =g*x^{\al}\sum_{k=0}^{\infty}\frac{(\lambda x^{\al+1})^{k}}{\Gamma((\al+1)k + (\al+1))}.
\]
\no Here and in whole paper by $*$ we denote the convolution on $(0,\infty)$, i.e. $f*g = \izx f(p)g(x-p)dp$. Finally, using this result together with (\ref{mla}) in (\ref{szereg1}) we obtain that the function $u$ defined by the formula
\eqq{
u(x) = \uz E_{\al+1}(\lambda x^{\al+1}) - g*x^{\al}E_{\al+1,\al+1}(\lambda x^{\al+1})
}{init}
is a solution to (\ref{zwyczajne}) with a boundary condition $u(0)= \uz$.
It remains to solve equation (\ref{zwyczajne}) with the zero condition in the endpoint of the interval. For this purpose, we take $x=1$ in (\ref{init}) and we obtain
\[
0=u(1) = \uz E_{\al+1}(\lambda) -   (g*y^{\al}E_{\al+1,\al+1}(\lambda y^{\al+1}))(1)
\]
and we may calculate $\uz$
\[
\uz =  (E_{\al+1}(\lambda))^{-1} (g*y^{\al}E_{\al+1,\al+1}(\lambda y^{\al+1}))(1).
\]
$\uz$ is well defined because, taking $\nu=\al+1, \mu=1$ in Proposition \ref{MLp} from Appendix, we obtain that $E_{\al+1}(\lambda)>0$ for $\lambda$ belonging to the sector $\vartheta$. Placing this $\uz$ in the formula (\ref{init}) we obtain the solution for (\ref{calkowe1}) which belongs to $\dd$
\[
u(x) =(E_{\al+1}(\lambda))^{-1} (g*y^{\al}E_{\al+1,\al+1}(\lambda y^{\al+1}))(1)E_{\al+1}(\lambda x^{\al+1}) - g*x^{\al}E_{\al+1,\al+1}(\lambda x^{\al+1}).
\]
That way we proved the lemma.
\end{proof}
In particular, we have shown that $R(I-\poch \da) = \ld$. This, together with the accretive property of $-\poch\da$ allows us to apply the Lumer-Phillips theorem which finishes the proof of Theorem \ref{C0}.
\end{proof}
\begin{rem}
  We note that the above theorem as well as other results from this paper are valid for the space interval $[0,L]$ for every fixed $0<L < \infty$ and $L=1$ was chosen just to simplify the notation.
\end{rem}
It remains to prove that the semigroup generated by $\poch\da$ can be extended to an analytic semigroup on a sector of complex plane.
\no Before we will prove that result, we need to formulate two auxiliary lemmas. A similar reasoning to the one carried in Lemma \ref{rown} may be found in \cite[Lemma 6]{Karkulik}.
\begin{lem}\label{rown}
The formulas $\norm{D^{\frac{1+\al}{2}}\ u}_{L^{2}(0,1)}$ and $\norm{u}_{H^{\frac{1+\al}{2}}(0,1)}$ define equivalent norms on~$\dd$.

\end{lem}
 \begin{proof}
Firstly, we will show that there exists a positive constant $c$ such that
\[
\norm{D^{\frac{1+\al}{2}}\ u}_{L^{2}(0,1)} \leq c \norm{u}_{H^{\frac{1+\al}{2}}(0,1)}.
\]
Using Proposition \ref{ie} we may write
\[
\norm{D^{\frac{1+\al}{2}}\ u}_{L^{2}(0,1)}  = \norm{I^{\frac{1-\al}{2}} u_{x}}_{L^{2}(0,1)} \leq c \norm{ u_{x}}_{H^{\frac{\al-1}{2}}(0,1)}.
\]
Due to Remark 12.8. \cite{Lions} we know that $\poch$ is a bounded and linear operator from $H^{s}(0,1)$ to $H^{s-1}(0,1)$ for $s \in [0,1] \setminus \{\frac{1}{2}\}$ thus
\[
\norm{D^{\frac{1+\al}{2}}\ u}_{L^{2}(0,1)}  \leq c \norm{ u}_{H^{\frac{\al+1}{2}}(0,1)}.
\]
To show the opposite inequality we notice that since $u \in \dd$ we have
\[
u(x)=-\int_{x}^{1} u_{x}(s)ds = -I_{-}u_{x} =- I_{-}^{\frac{1+\al}{2}}I_{-}^{\frac{1-\al}{2}}u_{x}
\]
and by Proposition \ref{eq} we may estimate
\[
\norm{ u}_{H^{\frac{\al+1}{2}}(0,1)} = \norm{I^{\frac{1+\al}{2}}_{-}I^{\frac{1-\al}{2}}_{-}u_{x}}_{{}^{0}H^{\frac{\al+1}{2}}(0,1)} \leq c\norm{I^{\frac{1-\al}{2}}_{-}u_{x}}_{L^{2}(0,1)}.
\]
Applying Proposition \ref{ie} and Proposition \ref{de} we may estimate further
\[
\norm{ u}_{H^{\frac{\al+1}{2}}(0,1)} \leq c \norm{ u_{x}}_{H^{\frac{\al-1}{2}}(0,1)} = c\norm{\p^{\frac{1-\al}{2}}I^{\frac{1-\al}{2}} u_{x}}_{H^{\frac{\al-1}{2}}(0,1)} \leq c \norm{I^{\frac{1-\al}{2}} u_{x}}_{L^{2}(0,1)},
\]
which finishes the proof.
\end{proof}
\begin{lem}
For $u \in \dd$ we have
\eqq{\Re(-\poch D^{\al}u,u) \geq \ca \norm{u}_{H^{\frac{1+\al}{2}}(0,1)}^{2}}{koer}
and
\eqq{\abs{(-\poch D^{\al}u,u)} \leq \ba \norm{u}_{H^{\frac{1+\al}{2}}(0,1)}^{2},}{bound}
where $\ca, \ba$ are positive constant which depends only on $\al$.

\end{lem}
\begin{proof}
We have already obtained in (\ref{koer0})
\[
\Re \left(-\poch \da u, u\right) \geq  \ca \norm{D^{\frac{1+\al}{2}}\ u}_{L^{2}(0,1)}^{2}.
\]
Hence, in order to prove (\ref{koer}) it is enough to apply the norm equivalence from Lemma~\ref{rown}. To show (\ref{bound}), we firstly notice that since $u \in \dd$, we know that $u_{x}\in {}_{0}H^{\al}(0,1)$ and from Corollary \ref{eqc} we infer that $\da u= I^{1-\al}u_{x} \in {}_{0}H^{1}(0,1)$.
Applying Proposition \ref{skladanie} in the first and third identity and $(\da u)(0)=0$ in the second one, we may write
\[
\poch \da u = \p^{\frac{1+\al}{2}}D^{\frac{1-\al}{2}} D^{\al}u =\p^{\frac{1+\al}{2}}\p^{\frac{1-\al}{2}} D^{\al}u=\p^{\frac{1+\al}{2}}D^{\frac{1+\al}{2}}u = \poch I^{\frac{1-\al}{2}}D^{\frac{1+\al}{2}}u.
\]
Integrating by parts, in view of $\overline{u}(1)=0$ we obtain
\[
\abs{(-\poch D^{\al}u,u)}=\abs{\izj \poch I^{\frac{1-\al}{2}}D^{\frac{1+\al}{2}}u \cdot \overline{u}dx} = \abs{\izj I^{\frac{1-\al}{2}}D^{\frac{1+\al}{2}}u \cdot \overline{u}_{x}dx}.
\]
Thus, by the identity (\ref{fub}) we get
\eqq{
\abs{(-\poch D^{\al}u,u)} = \abs{\izj D^{\frac{1+\al}{2}}u \cdot D^{\frac{1+\al}{2}}_{-}\overline{u}dx}
\leq \norm{D^{\frac{1+\al}{2}}u}_{L^{2}(0,1)}\norm{D^{\frac{1+\al}{2}}_{-}u}_{L^{2}(0,1)}.
}{dpdm}
Since $u(1)=0$, applying Proposition \ref{eq} we obtain that
\eqq{
\norm{D^{\frac{1+\al}{2}}_{-}u}_{L^{2}(0,1)} = \norm{\p^{\frac{1+\al}{2}}_{-}u}_{L^{2}(0,1)} \leq \ba \norm{u}_{^{0}H^{\frac{1+\al}{2}}(0,1)} =\ba \norm{u}_{H^{\frac{1+\al}{2}}(0,1)},
}{rownm}
where by $\ba$ we denote a positive constant dependent on $\al$.
Making use of this estimate and the norm equivalence from Lemma \ref{rown} in (\ref{dpdm}) we obtain the estimate~(\ref{bound}).
\end{proof}

\no Finally, we are ready to prove the main theorem.
\begin{theorem}\label{anal}
The operator $\poch \da: \dd \subseteq \ld  \rightarrow \ld$ generates an analytic semigroup.
\end{theorem}
\begin{proof}

We will give the proof of analyticity following the proof of \cite[Ch. 7, Theorem 2.7]{Pazy}, where the elliptic operators are studied.\\
At first, we notice that since $L^{2}(0,1)$ is a Hilbert space, the numerical range of $-\poch \da$ equals
\[
 S(-\poch \da)= \left\{\left(u, -\poch\da u\right): u \in \dd, \hd \norm{u}_{L^{2}(0,1)}=1 \right\}.
\]
We note that by (\ref{koer}) zero does not belong to $S(-\poch \da)$.
Let us denote $z=\left(u, -\poch\da u\right)$. Then, in view of (\ref{koer}) and (\ref{bound}), we obtain that
\[
\abs{\tan (\arg z)} = \abs{\frac{\Im z}{\Re z}} \leq \frac{\ba}{\ca},
\]
which implies
\[
S(-\poch \da) \subseteq \left\{\lambda \in \mathbb{C}: \abs{\arg \lambda} \leq \arctan\left(\frac{\ba}{\ca}\right)\right\}
\]
and $\arctan(\frac{\ba}{\ca}) < \frac{\pi}{2}$. We may choose $\nu$ such that $\arctan(\frac{\ba}{\ca})< \nu <\frac{\pi}{2}$ and denote $\Sigma_{\nu}:=\{\lambda:\abs{\arg \lambda} > \nu \}$. Then, $\Sigma_{\nu} \subseteq \mathbb{C} \setminus \overline{S(-\poch \da)}$ and there exists a positive constant $c_{\nu}$
such that
\[
d(\lambda, S(-\poch \da)) \geq c_{\nu} \abs{\lambda} \m{ for all }\lambda \in \Sigma_{\nu}.
\]
By Theorem \ref{C0} we know that $(-\infty,0] \subseteq \rho(-\poch \da)$, which implies that
\[
 \Sigma_{\nu} \cap \rho(-\poch \da) \neq \emptyset.
\]
We may apply Proposition \ref{numerical} from Appendix to the operator $-\poch \da$ to obtain that spectrum of $-\poch \da$ is contained in $\mathbb{C}\setminus \Sigma_{\nu}$, which means that $\Sigma_{\nu}\subseteq \rho (-\poch \da)$ and
\[
\norm{\left(\lambda I - (-\poch \da)\right)^{-1}} \leq \frac{1}{d(\lambda,\overline{S(\poch \da)})} \leq \frac{1}{c_{\nu}\abs{\lambda}} \m{ for all }\lambda \in \Sigma_{\nu}.
\]
Thus, the set $\{\lambda \in \mathbb{C}: \abs{\arg \lambda} < \pi - \nu\} \subseteq \rho(\poch \da)$ and
\[
\norm{\left(\lambda I - \poch \da \right)^{-1}} \leq \frac{1}{c_{\nu}\abs{\lambda}} \m{ for every }\lambda \in \mathbb{C}:  \abs{\arg \lambda} < \pi - \nu.
\]
Making use of \cite[Ch. 2, Theorem 5.2.]{Pazy} we obtain that the semigroup generated by $\poch \da$ can be extended to the analytic semigroup on the sector $\abs{\arg \lambda} < \pi - \nu$. That way we proved the theorem.
\end{proof}
\section{The general case}

The results of Theorem \ref{anal} may be extended for an operator $\poch (p(x)\da)$ if only the function $p$ satisfies appropriate assumptions. We state this result in the next theorem.

\begin{theorem} \label{gen}
Let $p \in W^{1,\infty}(0,1)$, be strictly positive, i.e. there exists $\delta$ such that $p \geq \delta >0$. Then, the operator $\poch (p(x)\da): \dd \subseteq \ld  \rightarrow \ld$ generates an analytic semigroup.
\end{theorem}

\no We split the operator $\poch (p(x)\da)$ into the sum $\poch (p(x) \da) = A_{1}+A_{2}$, where
\[
A_{1} = p(x)\poch\da, \hd A_{2} = p'(x)\da,
\]
\[
 D(A_{1}) := \dd, \hd D(A_{2}) := {}^{0}H^{1}(0,1): = \{u \in H^{1}(0,1),\hd u(1)=0\}.
\]
At first, we will show that the operator $A_{1}$ generates an analytic semigroup. In order to do it, we have to slightly modify the proof of Theorem \ref{anal}. Namely, we will prove the estimates similar to (\ref{koer}) and (\ref{bound}) for operator $A_{1}$. To obtain the estimate from below we will use the pointwise estimate from \cite{Al}.
\begin{lem}\cite[Lemma 1]{Al}\label{ali}
For any $\al \in (0,1)$ and any absolutely continuous function $f$ the following estimate holds
\[
(\da f)(x)f(x) \geq \frac{1}{2} (\da f^{2})(x) \hd a.e.
\]
\end{lem}
\begin{lem}\label{ajsz}
Let the constant $\ca$ come from estimate (\ref{koer}). Then, under the assumptions of Theorem \ref{gen} there exist positive constants $\bar{c},\bar{c}_{1}$ dependent only on $\delta,\al,\norm{p}_{W^{1,\infty}(0,1)}$ such that for any $u \in \dd$ the operator $A_{1}$ satisfies
\eqq{
\Re (-A_{1}u,u)\geq \frac{\delta \ca}{2}\norm{u}^{2}_{H^{\frac{1+\al}{2}}(0,1)} - \bar{c}\norm{u}_{L^{2}(0,1)}^{2}
}{ajd}
and
\eqq{\abs{(A_{1}u,u)} \leq \bar{c}_{1}\norm{u}^{2}_{H^{\frac{1+\al}{2}}(0,1)}.
}{ajg}
\end{lem}
\begin{proof}
In order to prove (\ref{ajd}) we will firstly obtain appropriate estimate for any real-valued function $v \in \dd$. Subsequently, we will make use of this estimate for $v= \Re u$ and $v = \Im u$. For a real-valued $v \in \dd$ we may estimate as follows
\[
-\izj p(x)\poch \da v \cdot v dx = -\delta\izj \poch \da v \cdot v dx - \izj (p(x)-\delta)\poch \da v \cdot v dx.
\]
The first term is estimated from below by $\delta \ca \norm{v}^{2}_{H^{\frac{1+\al}{2}}(0,1)}$ by estimate (\ref{koer}). Since $v \in \dd$ we have $\da v \in {}_{0}H^{1}(0,1)$ and we may integrate by parts the second term
\[
- \izj (p(x)-\delta)\poch \da v \cdot v dx = \izj p'(x) \da v \cdot v dx+  \izj (p(x)-\delta) \da v \cdot v_{x} dx =:I_{1}+I_{2}.
\]
Before we will estimate $I_{1}$ and $I_{2}$ we will show that for every $\al \in (0,1)$, $v \in \dd$ and every $\ve > 0$ there exists $c=c(\ve)$ such that
\eqq{
\norm{\da v}_{L^{2}(0,1)} \leq \ve \norm{v}_{H^{\frac{1+\al}{2}}(0,1)} + c(\ve)\norm{v}_{L^{2}(0,1)}.
}{dah}
Indeed, for $\al \in (\frac{1}{2},1)$ there exists $c$ dependent on $\al$ such that
\[
\norm{\da v}_{L^{2}(0,1)} = \norm{I^{1-\al} v_{x}}_{L^{2}(0,1)} \leq c \norm{v_{x}}_{H^{\al-1}(0,1)}
\leq c \norm{v}_{H^{\al}(0,1)},
\]
where we used Proposition \ref{ie} in the first inequality and \cite[Remark 12.8]{Lions} in the second one. Since the embedding of $H^{\frac{1+\al}{2}}(0,1)$ into $H^{\al}(0,1)$ is compact we obtain~(\ref{dah}) by a standard argument.
In the case $\al \in (0,\frac{1}{2}]$ we may choose $\beta \in (\frac{1-\al}{2},\frac{1}{2})$ and we get
\[
\norm{\da v}_{L^{2}(0,1)} = \norm{I^{1-\al} v_{x}}_{L^{2}(0,1)} \leq c \norm{I^{\beta} v_{x}}_{L^{2}(0,1)},
\]
where in this estimate we used \cite[Theorem 2.6]{Samko} and the constant $c$ depends on $\al$ and $\beta$. Since $\beta < \frac{1}{2}$ we may repeat the reasoning carried for the case $\al \in (\frac{1}{2},1)$ to obtain
\[
\norm{I^{\beta} v_{x}}_{L^{2}(0,1)}\leq c \norm{v_{x}}_{H^{-\beta}(0,1)}
\leq c \norm{v}_{H^{1-\beta}(0,1)}.
\]
Making use of compact embedding of $H^{\frac{1+\al}{2}}(0,1)$ into $H^{1-\beta}(0,1)$ we obtain (\ref{dah}) in the case $\al \in (0,\frac{1}{2}]$.
With a use of (\ref{dah}) we are able to estimate $I_{1}$ as follows
\[
\abs{I_{1}} \leq \norm{p'}_{L^{\infty}(0,1)}\norm{\da v}_{L^{2}(0,1)}\norm{v}_{L^{2}(0,1)} \leq \frac{\delta\ca}{4}\norm{v}_{H^{\frac{1+\al}{2}}(0,1)}^{2} + \bar{c}\norm{v}_{L^{2}(0,1)}^{2},
\]
where $\bar{c}$ is a positive constant dependent only on $\delta,\al,\norm{p}_{W^{1,\infty}(0,1)}$.
To estimate $I_{2}$ we note that $v_{x} = \p^{1-\al}\da v = D^{1-\al}\da v$. Recalling that $p-\delta \geq 0$ we may apply Lemma \ref{ali} to get
\[
I_{2} \geq \frac{1}{2}\izj (p(x)-\delta)D^{1-\al}(\da v)^{2}(x)dx = \frac{1}{2}\izj (p(x)-\delta)\poch I^{\al}(\da v)^{2}(x)dx,
\]
where in the last identity we used the fact that $\da v \in {}_{0}H^{1}(0,1)$. Integrating by parts we obtain
\[
I_{2} \geq \frac{1}{2}(p(1)-\delta) I^{\al}(\da v)^{2}(1) -\frac{1}{2} \izj p'(x)I^{\al}(\da v)^{2}(x)dx.
\]
We note that the first expression is nonnegative. We estimate the second one as follows
\[
\abs{\izj p'(x)I^{\al}(\da v)^{2}(x)dx} \leq \norm{p'}_{L^{\infty}(0,1)}\norm{I^{\al}(\da v)^{2}}_{L^{1}(0,1)} \leq c \norm{p'}_{L^{\infty}(0,1)}\norm{\da v}_{L^{2}(0,1)}^{2},
\]
where we applied \cite[Theorem 2.6]{Samko} and  the constant $c$ depends on $\al$. Making use of (\ref{dah}) with an appropriate $\ve$ we obtain that
\[
\abs{\frac{1}{2}\izj p'(x)I^{\al}(\da v)^{2}(x)dx}  \leq \frac{\delta\ca}{4}\norm{v}_{H^{\frac{1+\al}{2}}(0,1)}^{2} + \bar{c}\norm{v}_{L^{2}(0,1)}^{2},
\]
where $\bar{c}$ is a positive constant dependent only on $\delta,\al,\norm{p}_{W^{1,\infty}(0,1)}$. Combining estimates for $I_{1}$ and $I_{2}$ we finally obtain that
\[
-\izj p(x)\poch \da v \cdot v dx \geq \frac{\delta \ca}{2}\norm{v}^{2}_{H^{\frac{1+\al}{2}}(0,1)} - \bar{c}\norm{v}_{L^{2}(0,1)}^{2}.
\]
Applying this result for $v = \Re u$ and $v = \Im u$ we obtain
\[
\Re \left(-p(x)\poch \da u, u\right) = -\Re \izj p(x)(\poch \da u)(x) \cdot \overline{u(x)}dx
\]
\[
=- \izj p(x)\poch\da \Re u(x)\cdot \Re u(x)dx - \izj p(x)\poch \da \Im u(x)\cdot \Im u(x)dx
\]
\[
\geq \frac{\delta \ca}{2}\norm{u}^{2}_{H^{\frac{1+\al}{2}}(0,1)} - \bar{c}\norm{u}_{L^{2}(0,1)}^{2}.
\]
This way we proved (\ref{ajd}). In order to show (\ref{ajg}) we perform integration by parts

\[
\abs{(p(x)\poch\da u,u)} \leq \abs{\izj p(x)(\da u)(x)\cdot \overline{u}_{x}(x)dx} + \abs{\izj p'(x)(\da u)(x)\cdot \overline{u}(x)dx}
\]
\[
= \abs{\izj p(x)(I^{\frac{1-\al}{2}} D^{\frac{1+\al}{2}} u)(x)\cdot \overline{u}_{x}(x)dx} + \abs{\izj p'(x)(I^{\frac{1-\al}{2}} D^{\frac{1+\al}{2}} u)(x)\cdot \overline{u}(x)dx}
\]
Applying the Fubini Theorem we obtain further
\[
=\abs{\izj  (D^{\frac{1+\al}{2}}u)(x)\cdot (I^{\frac{1-\al}{2}}_{-} p \overline{u}_{x})(x)dx} + \abs{\izj  (D^{\frac{1+\al}{2}} u)(x)\cdot (I^{\frac{1-\al}{2}}_{-}p'\overline{u})(x)dx}
\]
\[
\leq \norm{D^{\frac{1+\al}{2}}u}^{2}_{L^{2}(0,1)} + \frac{1}{2}\norm{I^{\frac{1-\al}{2}}_{-}( p \overline{u}_{x})}^{2}_{L^{2}(0,1)}+\frac{1}{2}\norm{I^{\frac{1-\al}{2}}_{-}(p'\overline{u})}^{2}_{L^{2}(0,1)}.
\]
We note that
\[
\Gamma(\frac{1-\al}{2})(I^{\frac{1-\al}{2}}_{-} p \overline{u}_{x})(x) = \int_{x}^{1}(s-x)^{-\frac{\al+1}{2}}p(s)\overline{u}_{x}(s)ds
\]
\[
= \int_{x}^{1}(s-x)^{-\frac{\al+1}{2}}(p(s)-p(x))\overline{u}_{x}(s)ds + p(x)\int_{x}^{1}(s-x)^{-\frac{\al+1}{2}}\overline{u}_{x}(s)ds.
\]
Making use of regularity of $u$ and $p$ we may integrate by parts the first integral to obtain that
\[
\abs{\int_{x}^{1}(s-x)^{-\frac{\al+1}{2}}(p(s)-p(x))\overline{u}_{x}(s)ds} \leq c \norm{p'}_{L^{\infty}(0,1)}\int_{x}^{1}(s-x)^{-\frac{\al+1}{2}}\abs{u(s)}ds,
\]
where $c$ depends on $\al$. Hence,
\[
\norm{I^{\frac{1-\al}{2}}_{-}( p \overline{u}_{x})}^{2}_{L^{2}(0,1)} \leq c(\al)\left( \norm{p'}_{L^{\infty}(0,1)}^{2}\norm{u}_{L^{2}(0,1)}^{2} + \norm{p}_{L^{\infty}(0,1)}^{2}\norm{D^{\frac{\al+1}{2}}_{-}u}_{L^{2}(0,1)}^{2}\right).
\]
Finally, we note that due to \cite[Theorem 2.6]{Samko} we have
\[
\norm{I^{\frac{1-\al}{2}}_{-}(p'\overline{u})}^{2}_{L^{2}(0,1)} \leq \norm{p'}_{L^{\infty}(0,1)}^{2}\norm{I^{\frac{1-\al}{2}}_{-}\abs{\overline{u}}}^{2}_{L^{2}(0,1)}
\leq c(\al) \norm{p'}_{L^{\infty}(0,1)}^{2} \norm{u}_{L^{2}(0,1)}^{2}.
\]
Combining this results, applying Lemma \ref{rown} and the identity (\ref{rownm}) we arrive at (\ref{ajg}).
\end{proof}
\no Having established Lemma \ref{ajsz}, we may prove the following result.
\begin{lem}\label{ajan}
Under the assumptions of Theorem \ref{gen} the operator $A_{1}:\dd\rightarrow L^{2}(0,1)$ is a generator of analytic semigroup.
\end{lem}
\begin{proof}
We may show similarly as in the case of operator $\poch\da$ that for every $\lambda \geq 0$ we have $R(\lambda I - A_{1}) = L^{2}(0,1)$. Let us define $A_{1}^{\xi}:=A_{1}-\xi I$ for $\xi \geq \bar{c}$, where $\bar{c}$ comes from (\ref{ajd}). Then, we have $R(\lambda I - A^{\xi}_{1}) = L^{2}(0,1)$. Furthermore, from estimates (\ref{ajd}) and (\ref{ajg}) we deduce that there exist constants $\bar{c}_{2}, \bar{c}_{3}$ dependent only on $\al,\delta,\norm{p}_{W^{1,\infty}(0,1)}$ such that
\[
\Re (-A_{1}^{\xi}u,u) \geq \bar{c}_{2} \norm{u}^{2}_{H^{\frac{1+\al}{2}}(0,1)}, \hd \abs{(A_{1}^{\xi}u,u)} \leq \bar{c}_{3} \norm{u}^{2}_{H^{\frac{1+\al}{2}}(0,1)} \m{ for } u \in \dd.
\]
Hence, we may repeat the proof of Theorem \ref{anal} for an operator $A_{1}^{\xi}$ to obtain that it is a generator of an analytic semigroup. Thus, $A_{1}$ is also a generator of an analytic semigroup.
\end{proof}
\no At last, we are ready to prove Theorem \ref{gen}.
\begin{proof}[Proof of Theorem \ref{gen}]
In the case $p'\equiv 0$ a.e. on $(0,1)$ we obtain that $\poch p(x) \da = A_{1}$ and the claim follows from Lemma \ref{ajan}. In the case $p' \not \equiv 0$ we note that
\[
\norm{A_{2}u}_{L^{2}(0,1)} \leq \norm{p'}_{L^{\infty}(0,1)}\norm{\da u}_{L^{2}(0,1)} = \norm{p'}_{L^{\infty}(0,1)}\norm{I^{1-\al} u_{x}}_{L^{2}(0,1)}
\]
Since by \cite[Theorem 2.6]{Samko} the operator of fractional integration is bounded on $L^{2}(0,1)$ we may estimate further
\[
\norm{A_{2}u}_{L^{2}(0,1)} \leq c_{\al}\norm{p'}_{L^{\infty}(0,1)}\norm{u}_{{}^{0}H^{1}(0,1)}.
\]
Hence, we obtain that $A_{2} \in B({}^{0}H^{1}(0,1);L^{2}(0,1))$.
We note that ${}^{0}H^{1}(0,1) = [L^{2}(0,1),\dd]_{\frac{1}{1+\al}}$, hence by interpolation theorem there exists $c>0$ such that for every $u \in \dd$ there holds
\[
\norm{u}_{{}^{0}H^{1}(0,1)} \leq c \norm{u}_{L^{2}(0,1)}^{\frac{\al}{\al+1}}\norm{u}_{\dd}^{\frac{1}{\al+1}}.
\]
Thus, we may apply \cite[Proposition 2.4.1 (i)]{Lunardi} to obtain that the sum $A_{1}+A_{2}:\dd \rightarrow L^{2}(0,1)$ generates an analytic semigroup.
\end{proof}

\section{Applications}
In this section we will present a simple application of obtained results. We will investigate the solvability of the following parabolic-type problem
\begin{equation}\label{ab}
 \left\{ \begin{array}{ll}
u_{t} - \frac{\p}{\p x} \da u = 0 & \textrm{ in } (0,1) \times (0,T),\\
u_{x}(0,t) = 0, \ \ u(1,t) = 0 & \textrm{ for  } t \in (0,T), \\
u(x,0) = u_{0}(x) & \textrm{ in } (0,1). \\
\end{array} \right. \end{equation}
We will formulate the result in the theorem.
\begin{theorem}
If we assume that $\uz \in L^{2}(0,1)$, then there exists exactly one solution to (\ref{ab}) which belongs to $C([0,T];L^{2}(0,1)) \cap C((0,T];\dd) \cap C^{1}((0,T];L^{2}(0,1))$. Furthermore, there exists a positive constant $c$, such that the following estimate holds for every $t\in (0,T]$
\[
\norm{u(\cdot,t)}_{L^{2}(0,1)} + t\norm{u_{t}(\cdot,t)}_{L^{2}(0,1)} +t\norm{\poch\da u(\cdot,t)}_{L^{2}(0,1)} \leq c \norm{\uz}_{L^{2}(0,1)}.
\]
Moreover, $u \in C^{\infty}((0,T];L^{2}(0,1))$ and for every $t \in (0,T]$, for very $k\in \mathbb{N}$ we have $u(\cdot,t) \in D((\poch \da )^{k})\subseteq H^{(\al+1)k}_{loc}(0,1)$. The last inclusion implies that for every $t\in (0,T]$ $u(\cdot,t)~\in~C^{\infty}(0,1).$
\end{theorem}
\begin{proof}
Since, we know that the operator $\poch \da$ generates an analytic semigroup, we may apply to (\ref{ab}) the general semigroup theory. Then, the above result follows from \cite[Theorem 3.4]{Yagi} and \cite[Proposition 2.1.1]{Lunardi}.
\end{proof}
\begin{rem}
One may consider the problem (\ref{ab}) with nonzero right-hand-side. Then the solution is obtained by the variation of constants formula. We may also increase the regularity of solution to (\ref{ab}) assuming higher regularity of the initial condition in a standard way (see for instance \cite[chapter 3.2.]{Yagi}).
\end{rem}
\section{Appendix}
We collect here, the results from the general calculus which were used in the paper. At first, we would like to cite the result, concerning distribution of zeros of Mittag-Leffler function from \cite{ML}.
\begin{prop}\cite[Theorem 4.2.1]{ML}\label{MLp}
Let us recall the definitions of Mittag-Leffler functions
\[
E_{\nu,\mu}(z) := \sum_{n=0}^{\infty} \frac{z^{n}}{\Gamma(\mu + \nu n)}, \m{ with } \mu,\nu \in \mathbb{R},\hd \nu>0 \hd \m{ and } E_{\nu}(z):= E_{\nu,1}(z).
\]
If we suppose that either
\[
\nu < 1, \hd \mu \in [1,1+\nu] \m{ or }\hd \nu \in (1,2), \hd \mu \in [\nu-1,1]\cup[\nu,2],
\]
then all roots of the function $E_{\nu,\mu}$ lie outside the angle
\[
\abs{\arg z} \leq \frac{\pi \nu}{2}.
\]
\end{prop}

\no At last we quote here Theorem from \cite{Pazy}.

\begin{prop}\cite[Ch.1, Theorem 3.9.]{Pazy}\label{numerical}
Let $X$ be a Banach space. For a linear operator $A$ in $X$ we define its numerical range $S(A)$ as
\[
S(A) = \{\langle x^{*}, Ax\rangle: x \in D(A), \norm{x}=1, x^{*} \in X^{*}, \norm{x^{*}}=1,\langle x^{*}, x\rangle=1  \}
\]
Let us assume that $A$ is closed, linear and densely defined in $X$. We denote by $\Sigma:=\mathbb{C} \setminus \overline{S(A)}$. If $\lambda \in \Sigma$ then $\lambda I - A$ is injective and has closed range. Moreover, if $\Sigma_{0} \subseteq \Sigma$ is such that $\Sigma_{0}\cap \rho(A) \neq \emptyset$ then the spectrum of $A$ is contained in $\mathbb{C}\setminus \Sigma_{0}$ and
\[
\norm{\left(\lambda I - A\right)^{-1}} \leq \frac{1}{d(\lambda, \overline{S(A)})} \m{ for every } \lambda \in \Sigma_{0},
\]
where $d(\lambda, \overline{S(A)})$ is a distance of $\lambda$ from $\overline{S(A)}$.
\end{prop}
\section{Acknowledgments}
The author is grateful to Prof. Piotr Rybka for his inspiration to write this paper and to Dr. Adam Kubica, Prof. Piotr Rybka and Prof. Masahiro Yamamoto for their valuable remarks. The author would like to thank the Reviewer for a suggestion of adding a chapter concerning generalizations.
The author was partly supported by National Sciences Center, Poland through 2017/26/M/ST1/00700 Grant.

\end{document}